\documentclass{amsart}
\usepackage{graphicx} 
\usepackage{stmaryrd}
\usepackage{dsfont,color,bbm}
\graphicspath{ {./images/} }
\usepackage{hyperref}
\hypersetup{
    colorlinks=true,
    linkcolor=blue,
    filecolor=magenta,      
    urlcolor=cyan,
    pdftitle={Overleaf Example},
    pdfpagemode=FullScreen,
    }

\usepackage{amsmath,amsfonts}
\usepackage{amssymb}
\usepackage{amsthm}
\usepackage{mathrsfs}
\usepackage{dsfont}
\usepackage{array}
\usepackage{graphicx}
\usepackage{tensor}
\usepackage{tikz}
\usetikzlibrary{math}
\usepackage{xcolor}
\usepackage{verbatim}
\usepackage{enumitem}
\usepackage{etoolbox} 
\usepackage{constants}
\usepackage{esint}

\newtheorem{theorem}{Theorem}[section]
\newtheorem{proposition}[theorem]{Proposition}

\newtheorem{lemma}[theorem]{Lemma}
\newtheorem{definition}{Definition}[section]
\newtheorem{remark}{Remark}[section]

\newcommand{\R}{\mathbb{R}} 
\newcommand{\T}{\mathbb{T}} 

\newcommand{\Pc}{\mathcal{P}}

\newcommand{\de}{\mathrm{d}}
\DeclareMathOperator{\id}{id} 
\DeclareMathOperator{\Div}{div} 

\DeclareMathOperator{\Per}{Per}

\title{Existence of a solution of the TV Wasserstein gradient flow}

\author{Kexin Lin and Filippo Santambrogio}

\address{Universit\'e Lyon 1, Ecole Centrale de Lyon, INSA Lyon, Universit\'e Jean Monnet, CNRS, ICJ, UMR 5208, Villeurbanne, France}
\email{klin@math.univ-lyon1.fr, santambrogio@math.univ-lyon1.fr}

\begin{document}

\begin{abstract}
On the flat torus in any dimension we prove existence of a solution to the TV Wasserstein gradient flow equation, only assuming that the initial density $\rho_0$ is bounded from below and above by strictly positive constants. This solution preserves upper and lower bounds of the densities, and shows a certain decay of the BV norm (of the order of $t^{-1/3}$ for $t\to 0$ -- if $\rho_0\notin BV$, otherwise the BV norm is of course bounded -- and of the order of $t^{-1}$ as $t\to\infty$). This generalizes a previous result by Carlier and Poon, who only gave a full proof in one dimension of space and did not consider the case $\rho_0\notin BV$.

The main tool consists in considering an approximated TV-JKO scheme which artificially imposes a lower bound on the density and allows to find a continuous-in-time solution regular enough to prove that the lower bounds of the initial datum propagates in time, and study on this approximated equation the decay of the BV norm. 
\end{abstract}

\maketitle

\section{Introduction}
Variational schemes based on total variation (TV) are extremely popular in image processing for denoising purposes. The corresponding continuous-in-time equations, consisting in the \(L^2\) gradient flow of the TV functional have been widely studied (we refer, for instance, to Bellettini, Caselles and Novaga \cite{bellettini2002total}). Due to the current ubiquitous use of the Wasserstein metric in data analysis, a natural variant of the standard $L^2$ total variation denoising (also called ROF problem, see \cite{rudin1992nonlinear}) consists in replacing the data fidelity term with the Wasserstein distance $W_2$. We refer, for instance, to \cite{during2012high} for a first application of this ``static'' variational problem. The corresponding gradient flow, i.e.\ the gradient flow of the TV functional in the $W_2$ space, takes the form of a non-linear (and non-trivial) fourth-order parabolic PDE. 
The first comprehensive study of such a PDE can be found in \cite{carlier2019total}, where Carlier and Poon (see also later contributions in \cite{chambolle2023total}) provide a theoretical justification of the convergence of the JKO scheme under some conditions. More precisely, if \(\Omega\) is an open bounded connected subset of \(\R^d\) and \(\rho_0 \in BV(\Omega) \cap \mathcal{P}  (\Omega)\) is an initial datum bounded from below and above by two positive constants\footnote{Note the usual abuse of notation where measures are identified with their densities, when they are absolutely continuous. In particular, when we say that a probability measure is BV, or is bounded by some constants a.e., we mean that it is absolutely continuous and its density is BV or bounded by the same constants. In this paper essentially all the probability measures we use will be absolutely continuous.}, under the assumption that a minimum principle holds (i.e the minimal value of the density is preserved at each step of the scheme), then the TV-JKO scheme 
\[ 
    \rho_0^\tau = \rho_0, \quad \rho_{k+1}^{\tau} \in \textrm{argmin} \left\{ \frac{1}{2 \tau} W_2^2 (\rho_k^\tau, \rho) + J(\rho) , \; \rho \in BV(\Omega) \cap \mathcal{P}  (\Omega )\right\} 
\]
where \(J\) is the total variation functional, converges to a weak solution of the TV Wasserstein gradient flow:
\begin{equation}
    \label{eq:TV-gradient-flow}
    \partial_{t} \rho + \Div \left[ \rho \nabla \left( \Div \Big(\frac{\nabla \rho}{\vert \nabla \rho \vert}\Big)\right)\right] = 0, \; \rho |_{t=0} = \rho_0.
\end{equation}
    
They also conjectured that this minimum principle holds and proved it in dimension $d=1$. However, generalizing their proof to higher dimensions is challenging, since their proof was based on the geodesic convexity of the set of measures whose density is lower-bounded by a given constant, which does not hold for $d>1$ (equivalently, one can consider the geodesic convexity of the functional \(\int_\T H(\rho)\), where \(H(s) = s^q\) with negative \(q\), which is also specific to the one-dimensional case).

Here, we circumvent this issue by first studying a gradient flow of a modified energy which artificially imposes a lower bound on the densities, and then proving a minimum principle for this equation. Importantly, we prove such a minimum principle at the level of the continuous PDE, and not of the JKO scheme. 

More precisely, we consider the case where the domain is the flat $d$-dimensional torus \(\T^d\) (abbreviated as \(\T\)), and consider first an approximated TV-JKO scheme:
\[
    \rho_0^\tau = \rho_0, \quad \rho_{k+1}^{\tau} \in \textrm{argmin} \left\{ \frac{1}{2 \tau} W_2^2 (\rho_k^\tau, \rho) +F_\epsilon(\rho) \right\}, 
\]
where
\begin{align*}
    F_\epsilon (\rho) = J(\rho) + \epsilon\int f(\rho), \quad f(s) = \begin{cases}
        \frac{1}{s - c} & \text{if } s > c \\
        +\infty & \text{otherwise}.
    \end{cases}
\end{align*}

Since the term \(\epsilon\int f(\rho)\) provides an artificial lower bound \(\rho \geq c\) at the level of the JKO step, and since this was the only missing information in \cite{carlier2019total}, the JKO scheme is expected to converge, which allows us to obtain a weak solution of the approximated TV gradient flow:
\begin{equation}
    \label{eq:approximated_TV_gradient_flow}
    \partial_{t} \rho + \Div \left[ \rho \nabla \left( \Div \Big(\frac{\nabla \rho}{\vert \nabla \rho \vert}\Big) - \epsilon f'(\rho) \right)\right] = 0.
\end{equation}

In order to obtain our result, we will consider the limit \(\epsilon \to 0\) and we first want to establish a minimum principle for the solution of \eqref{eq:approximated_TV_gradient_flow}.

Denote \( \Delta_1 \rho = \Div \Big(\frac{\nabla \rho}{\vert \nabla \rho \vert}\Big)\). We consider the decay of a functional \(\int_\T H(\rho)\) along the flow, where \(H\) is a smooth convex function. Under some regularity assumptions, we can compute 
\begin{equation}
    \begin{aligned}
        \frac{d}{dt} \int_\T H(\rho) &= \int_\T H'(\rho) \partial_t \rho \label{eq:decay_formal_computation}\\
        &= -\int_\T H'(\rho) \Div \left[ \rho \nabla \left( \Delta_1 \rho - \epsilon f'(\rho) \right)\right] \\
        &= \int_\T \rho H''(\rho) \nabla \rho \cdot \nabla \left( \Delta_1 \rho - \epsilon f'(\rho) \right) \\
        &\leq \int_\T \rho H''(\rho) \nabla \rho \cdot \nabla \Delta_1 \rho,
    \end{aligned}
\end{equation}
if we choose \(g\) such that \(g'(s) = s H''(s)\), denote \(p = g(\rho)\), then using \(\Delta_1 \rho = \Delta_1 p\), we have 
\[
    \int_\T \rho H''(\rho) \nabla \rho \cdot \nabla \Delta_1 \rho = \int_\T \nabla p \; \nabla \Delta_1 p. 
\]

We need to prove that the right-hand side above is negative. This will be done rigorously later in the paper, and we will explain which regularity do we need to perform such a computation (see Lemma~\ref{lem:inequality_rho_div_z}). So far, we just give a formal intuition of the inequality.
Set \(G(\mathbf{v}) = \vert \mathbf{v} \vert\). Formally, the 1-Laplacian can be written as 
\[
    \Delta_1 p = \nabla \cdot \left( \nabla G (\nabla p) \right)
\]
and
\begin{align*}
    \int_\T \nabla p \; \nabla \Delta_1 p &= \int_\T \;  \sum_i p_i \left(\sum_j \Big(G_j (\nabla p)\Big)_{ji}\right) \\
    &= -\int_\T \;  \sum_{i, j, k} p_{ij} G_{ik}(\nabla p) p_{kj} \\
    &= -\int_\T D^2 p : D^2 G : D^2 p \leq 0 .
\end{align*} 

This means that all quantities of the form $\int H(\rho_t)$ are non-increasing in $t$ whenever $H$ is convex. Applied to $H(s)=s^q$ for positive $q$, we obtain that every $L^q$ norm is non-increasing and, at the limit $q\to\infty$, we can see that upper bounds are preserved; on the other hand one can use $H(s)=s^q$ for negative $q$ and, taking $q\to-\infty$, preserve the $L^\infty$ norm of $\rho^{-1}$, i.e.\ preserve the lower bounds on the density.  

The estimates performed above on the PDE only require the function $H$ to be convex; a corresponding technique, called {\it flow interchange} technique (developed by McCann, Matthes, and Savare \cite{matthes2009family}), exists to perform a similar estimate along the iterations of the JKO scheme, but this requires geodesic convexity of the functional $\rho\mapsto \int H(\rho(x))dx$. This is mainly a technical requirement in order to compare a first-order expansion with a true increment. Unfortunately, the case of negative exponent $q$ is not included in this geodesic convexity assumption, so we cannot obtain a lower bound on the JKO scheme but only, a priori, on the continuous PDE. Yet, the difficulty is that we need the lower bound in order to prove the convergence of the JKO scheme, in order to obtain a solution to the PDE. This means that before passing to the continuous world, we need a minimum principle at the JKO level. 
We will see that our extra term \(\epsilon\int_\T  f(\rho)\) provides not only a lower bound allowing to pass to the limit in the JKO scheme, but also ensures sufficient regularity to perform the computations we need, and to make the above formal computation rigorous.

In the study of gradient flows, a standard and simplifying assumption consists in requiring that the initial datum has finite energy. This means assuming, in our case, $\rho_0\in BV$. If we want to get rid of this assumption, we need to prove a strong decay of the BV norm which makes it bounded far from $t=0$ independently of the initial datum. This leads to the following formal computation
\begin{align*}
    \frac{d}{dt} J(\rho) = \frac{d}{dt} \int_\T \vert \nabla \rho \vert = & \int_\T \frac{\nabla \rho}{\vert \nabla \rho \vert} \cdot \nabla \partial_t \rho \\
    = & \int_\T \Delta_1 \rho \; \Div \left[ \rho \nabla \left( \Delta_1 \rho - \epsilon f'(\rho) \right)\right] \\
    = & -\int_\T \rho \vert \nabla \Delta_1 \rho \vert^2  + \epsilon \int_\T \rho \nabla \Delta_1 \rho \cdot \nabla f'(\rho) \\
    \leq & -\int_\T \rho \vert \nabla \Delta_1 \rho \vert^2.
\end{align*}
The final inequality holds thanks to Lemma \ref{lem:inequality_rho_div_z} which will be proved later. We use upper and lower bounds of $\rho$, and then we use either the Poincaré inequality on $\Delta_1\rho$, or a Gagliardo-Nirenberg inequality which exploits that $\Delta_1\rho$ is the divergence of a vector field of unit norm, and obtain (see Section \ref{sec:existence_TV_flow})
\[
    \frac{d}{dt} J(\rho) \leq - C J^2(\rho),  \quad \frac{d}{dt} J(\rho) \leq - C J^4(\rho).
\]

Both these estimates provide a decay for the BV norm: the first one gives $J(\rho_t)\leq Ct^{-1}$, the second $J(\rho_t)\leq Ct^{-1/3}$. These bounds do not depend  on the initial BV norm but only on the lower and upper bounds of the density.  This property implies that we can use approximation to get a solution of the TV Wasserstein gradient flow for any initial density bounded from below and above. Both inequalities are true and have different applications: for $t\to\infty$ the strongest one is $J(\rho_t)\leq Ct^{-1}$ (whose validity is strongly related to the fact that our equation is set on a bounded domain; otherwise it would not hold and only the other one would be true, which is consistent with some explicit examples studied in \cite{carlier2019total} where $J(\rho_t)=O(t^{-1/3})$ in the whole space), while for $t\to 0$ the strongest one is $J(\rho_t)\leq Ct^{-1/3}$ (the fact that the exponent is strictly larger than $-1$ is crucial in order to obtain the continuity of the curve at $t=0$).

The paper is organized as follows. In Section \ref{sec:preliminaries}, we introduce some properties of BV functions, of Wasserstein gradient flows and of the JKO scheme. In section \ref{sec:minimum_principle}, we establish a minimum principle for the approximated TV Wasserstein gradient flow under certain regularity conditions. In Section \ref{sec:existence_approximated_TV_flow}, we study the approximated TV-JKO scheme and prove existence and regularity of a solution. In Section \ref{sec:existence_TV_flow}, we pass to the limit $\varepsilon\to 0$ in the PDEs, in order to obtain a weak solution of the TV Wasserstein gradient flow when the initial density is in BV. Then we prove the decay of the BV norm and use it to get a solution of the TV Wasserstein gradient flow for any initial density which is only bounded from below and above.

\section{Preliminaries}
\label{sec:preliminaries}

\subsection{BV functions}

\begin{definition}
   The total variation of a function \(\rho \in L^1(\mathbb T)\) is given by
        \[
            J(\rho) = \sup \left\{ \int_\T -\rho \Div z: \; z \in C^\infty (\T, \R^N) , \Vert z \Vert_\infty \leq 1  \right\},
        \]
    and \(BV (\T) = \{ \rho \in L^1(\T),J(\rho) < \infty \}\).
\end{definition}

Let us introduce the spaces
\begin{align*}
    H^1_{\Div} (\T) = \{ z \in L^2(\T, \R^d): \Div z \in L^2(\T) \} \quad &\text{with} \; \Vert z \Vert_{H^1_{\Div}} = \Vert z \Vert_{L^2} + \Vert \Div z \Vert_{L^2}, \\
    H^2_{\Div} (\T) = \{ z \in L^2(\T, \R^d): \Div z \in H^1(\T) \} \quad &\text{with} \; \Vert z \Vert_{H^2_{\Div}} = \Vert z \Vert_{L^2} + \Vert \Div z \Vert_{H^1}. 
\end{align*}

It can be proved that if \(\rho \in BV(\T) \cap L^2 (\T)\), then the subdifferential of \(J\) at \(\rho\) is given by:
\[
    \partial_{L^2} J(\rho) = \{-\Div z : z \in H^1_{\Div} (\T), \; \Vert z \Vert_\infty \leq 1, \;J(\rho) \leq -\int_\T \rho \Div z \},
\]
see \cite{chambolle2010introduction}[2.2.1] for details. We denote 
\[
    \mathcal{A} = \{ (\rho, z); \; \rho \in BV(\T) \cap L^2(\T), \; z \in H^1_{\Div} (\T), \; \Vert z \Vert_\infty \leq 1, \; 
    J(\rho) \leq -\int_\T \rho \Div z \}.
\]

\begin{remark}
    The \(BV\) space can also be characterized as the set of those \(L^1\) functions whose distributional gradient is a finite vector Radon measure. We can write:
        \[
            J(\rho) = \int_\T \vert \nabla \rho \vert = \int_\T \langle z, D \rho\rangle \quad \text{if} \; (\rho,z ) \in \mathcal{A} .
        \]
\end{remark}

Next we introduce two lemmata concerning the relation between a density \(\rho\) and its corresponding optimal vector fields \(z\).

\begin{lemma}
    \label{lem:global_contrast_change}   
    Given a $C^1$ function $g$ such that $g'$ is bounded and strictly positive, then for any \((\rho, z) \in \mathcal{A}\), the pair \((g(\rho), z)\) is also in \(\mathcal{A}\).
\end{lemma}

\begin{proof}
    Using the same argument as in Proposition 3 in \cite{chambolle2016geometric}, we see that we have \((\rho, z) \in \mathcal{A}\) if and only if the level sets of \(\rho\) satisfy for every $t$,
    \[
    \quad \Per(\{\rho \geq t\}) = \int_{\{\rho \geq t\}} \Div z.
    \]
    Since \(g(\rho)\in BV(\T)\) and the level set of \(g(\rho)\) are the same as \(\rho\), we have \((g(\rho), z) \in \mathcal{A}\).
\end{proof}

The second lemma shows that under some regularity
condition, a certain inequality involving the derivatives of \(\rho\) and \(\Div z\) holds.

\begin{lemma}\label{DrhoDz}
    \label{lem:inequality_rho_div_z} 
    Given \((\rho, z) \in \mathcal{A}\), assume \(\rho \in H^1 (\T)\), \(z \in H^2_{\Div}(\T)\); then 
    \[
        \int_\T \nabla \rho \cdot \nabla \Div z \leq 0.
    \]
\end{lemma}

\begin{proof}
    Denote \(Z = \Div z\), and the assumption provides $Z \in H^1$. For any \(h \in \R^d\), set \(Z_h (x) = Z(x + h)\) and $j(h):=\int_\T \rho Z_h$. Then, we have 
    \[
        j(0) = - \int_\T \rho Z = \int |\nabla \rho|\geq -\int_\T \rho Z_h = j(h),
    \]
    so that $0$ is a maximal point of $j$.
    Thanks to the regularity of \(\rho\) and \(Z\), we have
    \[
        D j(h) = -\int_\T \rho \nabla Z_h =  -\int_\T \rho_{-h} \nabla Z, \quad D^2 j(h) = \int_\T \nabla \rho_{-h} \otimes \nabla Z.
    \]
    The maximality of \(0\) provides \(\mathrm{tr} \; D^2 j(0) \leq 0\), which proves the lemma.
\end{proof}

\subsection{Wasserstein gradient flows and the JKO scheme}

In this section, we give a heuristic explanation of the JKO scheme. We refer to  \cite{santambrogio2015optimal} and \cite{SurveyGF}, for a more detailed exposition.

Given a compact domain, here \(\T\), and a functional \(G: \mathcal{P} (\T) \to \R \cup \{+\infty\}\), we consider the iterated minimization scheme:
\[
    \rho^\tau_{k+1} \in \mathrm{argmin}_{\rho \in \Pc(\T)} \left\{ \frac{1}{2 \tau} W_2^2(\rho_k^\tau, \rho) + G(\rho)\right\}.
\]

Assuming that the first variation of \(\frac{\delta G}{\delta \rho}(\rho)\) of \(G\) exists, then any optimal measure \(\bar\rho\) in the above optimization problem satisfies
\[
    \frac{\delta G}{\delta \rho}(\bar\rho) + \frac{\varphi}{\tau} = \mathrm{const},
\]
where \(\varphi\) is the Kantorovich potential between \(\bar\rho\) and \(\rho^\tau_k\).

Notice that the optimal map \(T\) is equal to \(\id - \nabla \varphi\), so we have:
\[
    \frac{T - \id}{\tau} = - \frac{\nabla \varphi}{\tau} = \nabla \left( \frac{\delta G}{\delta \rho}(\bar\rho)  \right) .
\]

Denote \(\mathbf{v} = -\frac{T - \id}{\tau} = -\nabla \left( \frac{\delta G}{\delta \rho}(\bar\rho)  \right)\), which is an approximation of the velocity of the curve that one obtains at the limit. Thanks to the continuity equation, when \(\tau \to 0\), we expect to find a solution of 
\[
    \partial_t \rho - \nabla \cdot \left(\rho \nabla \left( \frac {\delta G}{\delta \rho}(\rho)  \right)\right) = 0,
\]
which is the Wasserstein gradient flow of the energy functional \(G\).

We also note that the norm of the velocity field $\mathbf{v}_t$ in $L^2(\rho_t)$ provides the so-called metric derivative $|\rho'|(t)$ of the curve $\rho$ in the space $W_2$ and that the gradient flow structure of the equation allows, by looking at the dissipation of $G$ along the curve $\rho$, to obtain 
\[
    \int_0^T \int_\T \left|\nabla \left( \frac {\delta G}{\delta \rho}(\rho_t) \right)\right|^2\de\rho_t\de t=\int_0^T|\rho'|(t)^2\de t\leq G(\rho_0)-G(\rho_T).  
\]
Whenever $G\geq 0$ the right-hand side can of course be bounded by $G(\rho_0)$, which only depends on the initial datum and provides useful uniform bounds. In particular, the $L^2$ integrability of the metric derivative implies $C^{0,\frac12}$ uniform continuity of the curve $\rho$ in $W_2$ metric.

In our case, the energy functional \(G\) is \(J(\rho) + \epsilon\int  f(\rho)\), hence we expect it will provide a solution of 
\[
    \partial_{t} \rho + \Div \left[ \rho \nabla \left( \Div \Big(\frac{\nabla \rho}{\vert \nabla \rho \vert}\Big) - \epsilon f'(\rho)\right)\right] = 0
\]
in the weak sense, where \(f\), we recall, is defined as
\[
    f(s) = \begin{cases}
        \frac{1}{s - c} & \text{if } s > c \\
        +\infty & \text{otherwise}.
    \end{cases}
\]
for some \(c > 0\). That is

\begin{definition}\label{defisol}
    A weak solution of \eqref{eq:approximated_TV_gradient_flow} is a function \(\rho \) such that
    \[\rho \in L^\infty((0, T); BV(\T)\cap L^\infty(\T)) \cap C^0([0, T]; (\Pc(\T), W_2)), \quad \rho > c \; \text{a.e.} \]
    and for which there exists \(z \in L^\infty((0,T) \times \T)\) satisfying 
    \begin{align*}
        &(\rho(t, \cdot), z(t, \cdot)) \in \mathcal{A} \quad \text{for a.e.\ } t \in (0, T), \\
        &\Div z - \epsilon f'(\rho) \in L^2((0, T); H^1(\T)),
    \end{align*}
    and such that the pair \((\rho, z)\) is a weak solution of
    \[
        \partial_{t} \rho + \Div \left[ \rho \nabla \left( \Div z - \epsilon f'(\rho) \right)\right] = 0; \quad \rho(0) = \rho_0,
    \]
    i.e.\ for any \(u \in C^\infty_c([0, T) \times \T)\), 
    \[
        \int_0^T \int_\T \left( \rho \partial_t u + \rho \nabla  (\Div z - \epsilon f'(\rho)) \cdot \nabla u  \right) \de x \de t= - \int_\T \rho_0(x) u(0, x) \de x.
    \]
\end{definition}

As for the TV Wasserstein gradient flow \eqref{eq:TV-gradient-flow}, we define its weak solution in a similar way, by replacing \(\epsilon f\) with \(0\). Note that in this case the condition $\Div z - \epsilon f'(\rho) \in L^2((0, T); H^1(\T))$ becomes $z\in L^2 ((0,T);H^2_{\Div}(\T))$.

\section{Minimum Principle for the approximated TV gradient flow}
\label{sec:minimum_principle}

Under some regularity conditions on \(\rho\) and \(\Div z\), we obtain a minimum principle for the approximated TV Wasserstein gradient flow.

\begin{theorem}
    \label{thm:regularity_to_minimum_principle}
    With a fixed \(\epsilon >0\), suppose that a pair \((\rho, z)\) with \(\rho > c\) is a weak solution of \eqref{eq:approximated_TV_gradient_flow} with initial condition \(\rho_0 \in \mathcal{P}  (\T) \cap L^\infty(\T)\), and that it satisfies the regularity assumptions
    \[
        \rho \in L^2((0, T); H^1(\T)), \quad z \in L^2((0, T); H^2_{\Div}(\T))
    \]
    then for any \(\alpha> c\), if \(\rho_0 \geq \alpha\), it follows that we also have \(\rho \geq \alpha\).
\end{theorem}

\begin{proof}[Proof of Theorem \ref{thm:regularity_to_minimum_principle}]
    We denote \(\rho_t = \rho(t, \cdot)\), considering it as a function of \(x\). Consider a smooth convex function \(H(s)\) defined on \(\R^+\), such that $H$ and $H'$ are globally Lipschitz continuous.

    From equation \eqref{eq:approximated_TV_gradient_flow}, we have \(\partial_t \rho \in L^2((0, T); H^{-1} (\T))\). Using the regularity assumption on \(\rho\), the computations made in \eqref{eq:decay_formal_computation} are rigorous, and we get the following inequality:
    \[
        \partial_t \int_\T H(\rho) \leq \int_\T \rho H''(\rho) \nabla \rho \cdot \nabla \Div z.
    \]
    Choose a suitable function \(g\in C^1(\R)\) satisfying the condition of Lemma \ref{lem:global_contrast_change} and \(\nabla g(\rho) = \rho H''(\rho) \nabla \rho\), then set \(p = g(\rho)\), we have  
    \[
        \partial_t \int_\T H(\rho_t) \leq \int_\T \nabla p_t \cdot \nabla \Div z_t \leq 0 ,
    \]
    thanks to Lemma \ref{lem:global_contrast_change} and \ref{lem:inequality_rho_div_z}. This shows that the integral $\int_\T H(\rho_t)$ is a non-increasing function of $t$ whenever $H$ is convex and $H,H'$ are Lipschitz. By monotone approximation, the same is also true for every convex function $H$. In particular for any \(q > 0\), we may take \(H_q(s) = s^{-q}\). We deduce
    \[
        \frac{\mathrm{d}}{\mathrm{d t}} \Vert 1/\rho \Vert_{L^q}^q \leq 0,
    \]
    and hence 
    \[
        \Vert 1/\rho_t \Vert_{L^q}^q \leq \Vert 1/\rho_0 \Vert_{L^q}^q .
    \]
    Taking the limit \(q \to \infty\), we obtain
    \[
        \frac{1}{\inf_\T \rho_t} \leq \frac{1}{\inf_\T \rho_0},
    \]
    which shows $\rho_t\geq\alpha$ and proves the desired minimum principle.
\end{proof}

\section{Existence in the approximated TV gradient flow for BV initial data}
\label{sec:existence_approximated_TV_flow}

In this section, for fixed \(\epsilon >0\), we study the approximated TV-JKO scheme with initial data \(\rho_0\) satifying 
\[
    \rho_0 \in \mathcal{P}  (\T) \cap BV(\T), \quad \rho_0 \geq \alpha \; \text{for some } \alpha > c.
\]

We will use the scheme to obtain a weak solution of the approximated TV Wasserstein gradient flow \eqref{eq:approximated_TV_gradient_flow} with sufficient regularity to apply Theorem \ref{thm:regularity_to_minimum_principle}.

\subsection{One step of the approximated TV-JKO scheme}
\label{sec:one_step}
\begin{equation}
    \label{eq:one_step}
    \rho_1 = \textrm{argmin}_{\rho \in \mathcal{P} (\T)}  \left\{ \frac{1}{2 \tau} W_2^2(\rho_0, \rho) + J(\rho) + \epsilon \int_\T f(\rho) \right\}.
\end{equation}

\subsubsection{Euler-Lagrange equation}

Firstly, we follow the proof strategy of \cite{carlier2019total}[Proposition 3.1], which proves that the minimizer of the following minimization problem:
\[
    \min_{\rho \in \mathcal{P} (\T)}  \left\{ \frac{1}{2 \tau} W_2^2(\rho_0, \rho) + J(\rho) + \epsilon \int_\T \rho \log \rho \right\}
\]
is bounded away from \(0\). The lower bound depends on \(\tau\), but it is enough so as to compute the optimality conditions charcterizing such a minimizer in terms of an Euler-Lagrange equation. Note that also in the above problem the quantity to minimize includes a term of the form $\epsilon \int_\T f(\rho)$, and the optimal $\rho$ is bounded away from $0$ as a consequence of $f'(0)=-\infty$. In \eqref{eq:one_step}, since $f'(c)=-\infty$, an easy adaptation of the same argument proves that the optimal \(\rho_1\) stays uniformly away from \(c\) (again, the lower bound depends on \(\tau\)). The point to be adapted is the construction in \cite{carlier2019total}[Proposition 3.1, Figure 2], and we build now a competitor that remains separated from \(c\) and compute the corresponding energy difference. The differences estimates for the Wasserstein term and the total variation term are identical to those in \cite{carlier2019total}[(3.6), (3.7)]. Moreover, since \(f\) is decreasing, the relevant estimation needed for the potential energy term are even simpler to compute, which allows us to prove the desired result.

Secondly, we observe that strict positivity of \(\rho_1\) ensures the existence of a unique Kantorovich potential in the transport problem from $\rho_1$ to $\rho_0$, and hence of a first variation for \(W^2_2(\rho_0,\cdot)\), see \cite[Section 7.2.2]{santambrogio2015optimal} for details. Since \(\rho_1\) stays away from \(c\), the first variation for functional \(\int f (\rho) \) at \(\rho_1\) is equal to \(f'(\rho_1) \). Since both term \(W^2_2(\rho_0,\cdot)\) and \(\int f (\rho) \) are differentiable, we have the following optimality condition which characterizes \(\rho_1\):
\begin{equation}
    \label{eq:one_step_characterization}
    \frac{\varphi_1 }{\tau} - \Div z_1 + \epsilon f'(\rho_1)= 0 \quad \text{a.e.\ in } \T,
\end{equation}
where \(\varphi_1\) is the Kantorovich potential between \(\rho_1\) and \(\rho_0\), and \((\rho_1, z_1) \in \mathcal{A}\).

\subsubsection{Maximum principle}
\label{sec:maximum_principle}
Let \(K = \{\rho \in \mathcal{P} (\T): \rho \leq \Vert \rho_0 \Vert_{L^\infty(\T)} \; a.e.\}\) be the set of probability measures which are absolutely continuous, with density bounded from above by $\Vert \rho_0 \Vert_{L^\infty(\T)}$. We consider the \(W_2\) projection of \(\rho_1\) to \(K\), denoted as \(\hat{\rho}_1\). By Theorem 4.2 and Remark 4.3 in \cite{carlier2019total}, all the three terms in \eqref{eq:one_step} decrease when replacing \(\rho_1\) by \(\hat{\rho}_1\). The fact that \(\rho_1\) is the unique minimizer (uniqueness comes, for instance from the strict convexity of $f$) implies \(\rho_1 \in K\) i.e.\ \(\rho_1 \leq \Vert \rho_0 \Vert_{L^\infty(\T)}\). 

\subsubsection{Regularity}

\begin{theorem}
    \label{thm:one_step_regularity}
    \(\rho_1, f'(\rho_1)\) are in \(H^1(\T)\) with the estimate:
    \begin{align}
        \label{eq:estimate_1}
        \Vert \nabla f'(\rho_1) \Vert_{L^2} &\leq \frac{1}{\epsilon} \Vert \nabla \left[-\Div z_1 + \epsilon f'(\rho_1) \right] \Vert_{L^2}, \\
        \label{eq:estimate_2}
        \Vert \nabla \Div z_1 \Vert_{L^2} &\leq \Vert \nabla \left[-\Div z_1 + \epsilon f'(\rho_1) \right] \Vert_{L^2}.
    \end{align}
\end{theorem}

Before proving Theorem \ref{thm:one_step_regularity}, we first introduce the ROF problem, as introduced by Rudin, Osher and Fatemi in \cite{rudin1992nonlinear},
\begin{equation}
    \label{eq:rof}
    \bar{u} = \textrm{argmin}_{u \in BV (\T)}  \left\{ J(u) + \frac{1}{2} \int_\T (u - g)^2 \right\}, \quad \text{where}\; g \in L^2  \; \text{is given}.
\end{equation}
The optimality conditions for this convex optimization problem are
$$-\Div \bar{z} + \bar{u}=g \quad \mbox{ for }(\bar{u}, \bar{z})\in \mathcal A$$
(in the sense that $\bar{u}$ is optimal if and only if there exists $\bar{z}$ such that $(\bar{u}, \bar{z})\in \mathcal A$ and the above equality is satisfied).

We have the following regularity result for the ROF problem.

\begin{proposition}
    \label{prop:rof_H1_regularity}
    Given \(g \in H^1(\T)\), the solution of \eqref{eq:rof} belongs to \(H^1(\T)\). Futhermore, we have the estimate:
    \[
        \Vert \nabla \bar{u} \Vert_{L^2} \leq \Vert \nabla g \Vert_{L^2}.
    \]
\end{proposition}

\begin{proof}
    We will prove this proposition using regularity-via-duality as in \cite{santambrogio2018regularity}.

    First, let us introduce the functionals \(P\) and \(Q\) as follows:
    \begin{align*}
        P(u) &= J(u) + \frac{1}{2} \int_\T u^2  - \int_\T g u \; ,\\
        Q(v) &= \frac{1}{2} \int_\T (g -v )^2.
    \end{align*}
    Thanks to the equality
    \[
        \frac{1}{2} a^2 + \frac{1}{2} b^2 =  ab  + \frac{1}{2} \vert a- b \vert^2,
    \]
    we have
    \[
        P(u) + Q(v) = J(u) + \int_\T - uv + \frac{1}{2} \int_\T \vert u + v - g\vert^2.
    \]
    Hence, using the conditions on $\bar{u}$ and on the corresponding $\bar{z}$, we have
    \[
        P(\bar{u}) + Q(-\Div \bar{z}) = 0.
    \]
    Setting \(\bar{u}_h(x) = \bar{u}(x+h)\), we have
    \begin{align*}
        P(\bar{u}_h) - P(\bar{u}) &= P(\bar{u}_h) + Q(-\Div \bar{z}) \\
        &= J(\bar{u}_h) + \int_\T  \bar{u}_h \Div \bar{z} + \frac{1}{2} \int_\T \vert \bar{u}_h - \bar{u} \vert^2 \\
        &\geq \frac{1}{2} \int_\T \vert \bar{u}_h - \bar{u}\vert^2. 
    \end{align*}
    We then note that, because of translation invariance, the difference $ P(\bar{u}_h) - P(\bar{u})$ equals $\int_\T g (\bar{u} - \bar{u}_h) $. The function $j(h):=P(\bar{u}_h)$ is minimal at $h=0$ and can be written as $j(h)=c-\int_\T g\bar{u}_h=c-\int_\T g(x-h)\bar{u}(x)dx$. We can then compute $\nabla j(h)=\int_\T \nabla g(x-h)\bar{u}(x)dx=\int_\T \nabla g \; \bar{u}_h$ and, by minimality, we have $\nabla j(0)=0$. This in turns implies 
    $$\nabla j(h)=\int_\T (\nabla g)(\bar{u}_h-\bar{u}),\quad\mbox{ and }|\nabla j(h)|\leq ||\nabla g||_{L^2}||\bar{u}_h-\bar{u}||_{L^2}.$$

    We then consider $\omega(r):\sup_{\eta\,:\,|\eta|\le r} \Vert \bar{u}_\eta - \bar{u} \Vert_{L^2}, $ and take an arbitrary $h$ such that $|h|\le r$. We have
    \[
        \frac{1}{2} \Vert \bar{u}_h - \bar{u} \Vert_{L^2}^2 \leq j(h) - j(0) \leq \vert h \vert \Vert \nabla g \Vert_{L^2} \sup_{\eta \in [0, h]} \Vert \bar{u}_\eta - \bar{u} \Vert_{L^2}\le \vert h \vert \Vert \nabla g \Vert_{L^2} \omega(r).
    \]
    This implies
    \[  
       \frac12 \omega(r)^2 \leq C r\omega(r),
    \]
    and hence $\omega(r) \leq 2C r$.
    This estimate exactly means \( \bar{u} \in H^1\). 

    We now want to obtain a more precise estimate on the $L^2$ norm of \(\nabla \bar{u}\). We use again the Euler Lagrange equation of \eqref{eq:rof}, i.e.\ \(-\Div \bar{z} + \bar{u} = g\). Since all the terms are $H^1$ (the term $\Div \bar{z}$ is $H^1$ since the two other terms are $H^1$), we differentiate the equation, then multiply by \(\nabla \bar{u}\) both sides and integrate over \(\T\), 
    \[
    - \int_\T \nabla \bar{u} \cdot \nabla \Div \bar{z} + \vert \nabla \bar{u} \vert^2 = \int_\T \nabla  \bar{u} \cdot \nabla g,
    \]
    using Lemma \ref{lem:inequality_rho_div_z}, we have \(- \int_\T \nabla \bar{u} \cdot \nabla \Div \bar{z} \geq 0\), which gives
    \[
        \int_\T \vert \nabla \bar{u} \vert^2 \leq \int_\T \nabla \bar{u} \cdot \nabla g,
    \]
    which implies
    \[
        \Vert \nabla \bar{u} \Vert_{L^2} \leq \Vert \nabla g \Vert_{L^2}. \quad \qedhere
    \]
\end{proof}

\begin{remark}
    \label{rem:rof_gradient_estimate}
    We could have used the inequalities that we used to obtain  \(\bar{u} \in H^1\) in order to estimate \(\Vert \nabla \bar{u} \Vert_{L^2}\) but this would have provided $||\nabla \bar{u}||_{L^2}\leq 2||\nabla g||_{L^2}$, which is less precise.
\end{remark}

\begin{remark}
    In the (ROF) problem, besides the \(H^1\) estimate which we showed here, one can also prove that any modulus of continuity of $g$ is actually inherited by $\bar{u}$ (see \cite{chambolle2010introduction} for a proof).
\end{remark}

\begin{proof}[Proof of Theorem \ref{thm:one_step_regularity}]
    We set \(u_1 := \epsilon f'(\rho_1)\). Since \(\rho_1\in BV\)  and $\rho_1$ stays away from \(c\), it follows that \(u_1 \in BV\). 

    By Lemma \ref{lem:global_contrast_change}, we can rewrite the Euler Lagrange equation \eqref{eq:one_step_characterization} as \((u_1, z_1) \in \mathcal{A}\) together with the equation
    \begin{equation}
        \label{eq:one_step_characterization_u}
        -\Div z_1 + u_1 = g_1,
    \end{equation}
    where \(g_1 = -\frac{\varphi_1 }{\tau}\) is Lipschitz and hence belongs to \(H^1\).

    Hence \(u_1\) is the solution of a ROF problem with datum $g_1$ and using Proposition \ref{prop:rof_H1_regularity}, we have \(u_1\) is in \(H^1\) and the estimate \eqref{eq:estimate_1} holds. We then deduce \( f'(\rho_1), \rho_1 \in H^1\) and $z_1\in H^2_{\Div}$, which allows us to use Lemma \ref{DrhoDz}. This in turn allows to obtain the estimate \eqref{eq:estimate_2}, i.e.\ we differentiate the equation \eqref{eq:one_step_characterization}, then multiply by \(\nabla \Div z_1\) both sides and integrate over \(\T\), 
    \[
        - \int_\T \vert \nabla \Div z_1 \vert^2 + \epsilon \nabla \Div z_1 \cdot \nabla f'(\rho_1)= \int_\T \nabla \Div z_1 \cdot \nabla g_1,
    \]
    using Lemma \ref{lem:inequality_rho_div_z}, we have \(\int_\T \nabla \Div z_1 \cdot \nabla f'(\rho_1) \leq 0\), which gives
    \[
        \int_\T \vert \nabla \Div z_1 \vert^2 \leq - \int_\T \nabla \Div z_1 \cdot \nabla g_1.
    \]
    This implies
    \[
        \Vert \nabla \Div z_1 \Vert_{L^2} \leq \Vert \nabla g_1 \Vert_{L^2}. \quad \qedhere
    \]
\end{proof}

\subsection{Convergence of the approximated TV-JKO scheme}
\label{sec:approx_JKO_flow}

We construct the approximated TV-JKO scheme as follows:
\begin{equation}
    \label{eq:rho_approximated_TV-JKO-scheme}
    \rho_0^\tau = \rho_0, \quad \rho_{k+1}^{\tau} \in \textrm{argmin} \left\{ \frac{1}{2 \tau} W_2^2(\rho^\tau_k, \rho) + J(\rho) + \epsilon \int_\T f(\rho) \right\} .
\end{equation}

Using optimality conditions for \(\rho^\tau_{k+1}\), there exists \(z^\tau_{k+1}\) such that 
\begin{equation}
    \label{eq:z_approximated_TV-JKO-scheme}
    \frac{\varphi^\tau_{k+1} }{\tau} - \Div z^\tau_{k+1} + \epsilon f'(\rho^\tau_{k+1})= 0 \quad \text{a.e.\ in } \T,
\end{equation}
where \(\varphi^\tau_{k+1}\) is the Kantorovich potential between \(\rho^\tau_{k+1}\) and \(\rho^\tau_k\) and \((\rho^\tau_{k+1}, z^\tau_{k+1}) \in \mathcal{A}\). 

Since \(\rho_0 \in L^\infty\), we have a uniform \(L^\infty\) bound for all \(\rho^\tau_k\) thanks to the one step maximum principle \ref{sec:maximum_principle}, and from our artificial lower bound, we have 
\begin{equation}
    \label{eq:rho_upper_lower_bound}
    M := \Vert \rho_0 \Vert_{L^\infty} \geq \rho^\tau_k > c , \quad a.e. 
\end{equation}

Applying Theorem \ref{thm:one_step_regularity} to the approximated JKO scheme, we have 
\begin{align}
    \label{eq:JKO_step_regularity_f_prime}
    \Vert \nabla f'(\rho_k^\tau) \Vert_{L^2} &\leq \frac{1}{\epsilon} \Vert \nabla \left[ -\Div z_k^\tau + \epsilon f'(\rho_k^\tau) \right]\Vert_{L^2}, \\
    \label{eq:JKO_step_regularity_div_z}
    \Vert \nabla \Div z_k^\tau \Vert_{L^2} &\leq \Vert \nabla \left[ -\Div z_k^\tau + \epsilon f'(\rho_k^\tau) \right]\Vert_{L^2},
\end{align}
which implies
\begin{equation}
    \label{eq:JKO_step_regularity_rho}
    \Vert \nabla \rho_k^\tau \Vert_{L^2} \leq \frac{C}{\epsilon} \Vert \nabla \left[ -\Div z_k^\tau + \epsilon f'(\rho_k^\tau) \right]\Vert_{L^2},
\end{equation}
where $C=(\inf \{f''(s)\,:\, s\in (c,M]\})^{-1}$.

We extend \(\{\rho^\tau_k\}, \{z^\tau_k\}\) by piecewise constant interpolation. We have the following bound:
\begin{equation}
    \label{eq:J_uniform_boundness_epsilon}
    \begin{aligned}
        &\frac{1}{2\tau} \sum W^2_2 (\rho^\tau_k, \rho^\tau_{k+1} ) \leq F(\rho_0); \\
        &\sup_{t \in [0, T]} J(\rho^\tau(t, \cdot))\leq \sup_{t \in [0, T]} F(\rho^\tau(t, \cdot)) \leq F(\rho_0).
    \end{aligned}
\end{equation} 

Using the fact that the embedding of \(BV (\T)\) into \(L^p(\T)\) is compact for every \(p \in [1, \frac{d}{d-1})\), and the Lebesgue dominated convergence theorem as well as \eqref{eq:rho_upper_lower_bound}, up to a suitable subsequence, we may assume that 
\[
    \rho^\tau \to \rho \ \text{a.e.\ in } (0,T)\times\T \quad\text{and in } L^p((0,T)\times\T) , \; \forall p<\infty .
\]

Using the same argument as in \cite{santambrogio2015optimal}[8.3], up to a further suitable subsequence, we may assume that
\[ 
    \sup_{t \in [0, T]} W_2 (\rho^\tau(t), \rho(t)) \to  0 \; \text{as} \; \tau \to 0
\]
for some limit curve \(\rho \in C^{0, \frac12} ((0, T); (\Pc(\T), W_2)) \cap L^\infty((0, T); BV(\T) \cap L^\infty(\T))\) and \(\rho \geq c\).

Thanks to the lower bound $\rho^\tau\geq c$, we have
\begin{equation}
    \label{eq:zf_uniform_boundness_epsilon}
    \int_0^T \Vert \nabla (\Div z^\tau - \epsilon f'(\rho^\tau)) \Vert^2_{L^2(\T)} \leq C F(\rho_0) .
\end{equation}
This inequality is classical in gradient flow theory, but the \(L^2\) norm on the left-hand side is usually weighted with \(\rho^\tau\), which we can avoid here because of \(\rho^\tau > c\).

Using \eqref{eq:JKO_step_regularity_div_z} and \eqref{eq:zf_uniform_boundness_epsilon}, since \(\Div z^\tau\) has zero mean, we may assume that (up to further subsequences) that there is some \(z \in L^\infty((0, T)\times \T) \cap L^2( (0, T); H^2_{\Div} (\T))\) such that \(z^\tau\) converges to \(z\) weakly-\(\ast\) in \(L^\infty((0, T)\times \T)\) and \((\Div z^\tau, \nabla \Div z^\tau)\) converges to \((\Div z, \nabla \Div z)\) weakly in \(L^2((0, T)\times \T)\). The compatibility of \((\rho(t, \cdot) \) and \(z(t, \cdot))\) can be shown using the same argument as in \cite{carlier2019total} or using Lemma \ref{compazrho}

Using \eqref{eq:JKO_step_regularity_rho} and \eqref{eq:zf_uniform_boundness_epsilon}, with respect to \(\tau\), we obtain a uniform  $L^2$ bound on $\nabla \rho^\tau$, and since the $L^1$ norm of $\rho^\tau$ is fixed, we also deduce a uniform  $L^2((0, T); H^1 (\T))$ bound on $\rho^\tau$. We also want a uniform $H^1$ bound on $f'(\rho^\tau)$ and, again, its gradient is bounded in $L^2$. In order to bound its $L^2$ norm we use the following facts:
\begin{enumerate}
    \item From $F(\rho^\tau(t))\leq F(\rho_0)$, we deduce a uniform bound $\int_\T \frac{1}{\rho^\tau(t)-c}\leq C_0$. 
    \item From a simple Markov inequality we deduce that the set $\{x\in\T\,:\,\rho^\tau(t)\leq c+\frac{1}{2C_0}\} $ has measure at most $1/2$.
    \item We then deduce that the functions $v^\tau:=(f'(\rho^\tau(t))-f'(c+\frac{1}{2C_0}))_-$ vanish on a set with measure at least $1/2$, for every $t$. 
    \item From a suitable Poincaré inequality, we deduce the estimate on \( ||v^\tau(t)||_{L^2} \), i.e.\ $||v^\tau(t)||_{L^2}\leq C||\nabla v^\tau(t)||_{L^2}\leq C||\nabla (f'(\rho^\tau(t)))||_{L^2}$.
    \item Using $f'\leq 0$ we deduce that $f'(\rho^\tau)-v^\tau$ is bounded by $|f'(c+\frac{1}{2C_0})|$ and this allows to obtain a bound on $||f'(\rho^\tau)||_{L^2((0,T) \times \T)}$.
\end{enumerate}
Hence, up to extracting a further subsequence, we may assume that \(f'(\rho^\tau)\) converges to \(f'(\rho)\) weakly in \(L^2((0, T); H^1(\T))\). 

The limiting equation \eqref{eq:approximated_TV_gradient_flow} can be derived using standard computations in Wasserstein gradient flows, as it is done, for instance, in \cite{carlier2019total}[Theorem 5.2].

Before stating the main result of this section, we prove the Lemma we needed for the compatibility of $\rho$ and $z$.

\begin{lemma}\label{compazrho}
    Suppose that a sequence $(\rho_n,z_n)$ of functions on $[0,T]\times\T$ satisfies 
    \begin{enumerate}
        \item $(\rho_n(t), z_n(t)) \in \mathcal{A}$ for a.e.\ $t\in [0,T]$,
        \item $\rho_n\to \rho$ in $L^2([0,T]\times\T)$ and $z_n\to z$ weakly-\(\ast\) in $L^\infty([0,T]\times\T)$,
        \item $z_n$ is bounded in $L^2([0,T];H^1_{\Div}(\T))$,
    \end{enumerate}
    then we also have $(\rho(t), z(t)) \in \mathcal{A}$ for a.e.\ $t\in [0,T]$.
\end{lemma}
\begin{proof}
   We use the Fatou lemma, the fact that, up to a subsequence, for a.e.\ $t$ we have $\rho_n(t)\to \rho(t)$ in $L^2(\T)$, and the lower semi-continuity of \(J\), in order to obtain:
\begin{equation}
    \label{eq:J_compatibility}
    \begin{aligned}
        \int_0^T J(\rho_t) &\leq \int_0^T \liminf_{\epsilon \to 0} J(\rho_\epsilon(t)) \\
        &\leq \liminf_{\epsilon \to 0} \int_0^T J(\rho_\epsilon(t)) \\
        &\leq \liminf_{\epsilon \to 0} -\int_0^T\int_\T \rho_\epsilon\Div z_\epsilon \\
        &=- \int_0^T\int_\T \rho \Div z.
    \end{aligned}
\end{equation}
The last equality is due to the strong convergence of \(\rho_\epsilon\) and the weak convergence of \(\Div z_\epsilon\) in \(L^2((0, T) \times \T)\), which finishes the proof. 
\end{proof}

We then deduce the following existence theorem.

\begin{theorem}
    \label{thm:approximated_JKO_to_regularity}
    Given an initial density \(\rho_0 \in \mathcal{P}  \cap L^\infty \cap BV\) such that $\mathrm{ess}\inf \rho_0>c$, there exists a vanishing sequence of time steps \(\tau_n \to 0\) such that the limit of \(\rho^\tau\), denoted as \(\rho\), belongs to the class \(L^2((0, T); H^1(\T))\), and the limit of \(z^\tau\), denoted as \(z\),  belongs to the class \(L^2((0, T); H^2_{\Div}(\T))\), and the pair \((\rho, z)\) is a weak solution of the approximated TV gradient flow \eqref{eq:approximated_TV_gradient_flow}.
\end{theorem}

\section{Existence of the TV Wasserstein gradient flow and BV decay}
\label{sec:existence_TV_flow}

We start from the case where the initial density is BV.

\begin{theorem}\label{eps0BV}
     Given an initial density \(\rho_0 \in \mathcal{P} (\T) \cap L^\infty (\T) \cap BV(\T)\), with $\rho_0\geq\alpha>c$, there exists a continuous curve of measures \(\rho\) which is a weak solution of the TV gradient flow \eqref{eq:TV-gradient-flow} in the sense of Definition \ref{defisol} and which satisfies $J(\rho_t)\leq J(\rho_0)$ and $\alpha\leq \rho\leq ||\rho_0||_{L^\infty}$.
\end{theorem}

\begin{proof}
We call \(\rho_\epsilon\) the solution obtained by the approximated TV-JKO scheme with the extra energy term \(\epsilon f\) and the initial data \(\rho_0\). 

Using Theorem \ref{thm:regularity_to_minimum_principle}, we have \(\rho_\epsilon \geq \alpha\).

We now want to pass to the limit as $\epsilon\to 0$. We use the bounds that we proved, in particular \eqref{eq:rho_upper_lower_bound}, which provides a uniform $L^\infty$ bound, \eqref{eq:J_uniform_boundness_epsilon}, which provides a bound on $J$ and at the same time the equicontinuity (valued in $W_2$) of the curves $\rho_\epsilon$, and \eqref{eq:zf_uniform_boundness_epsilon}, which, together with \eqref{eq:JKO_step_regularity_div_z}, provides a uniform $H^1$ bound on $\Div z_\epsilon$.  All these bounds are independent of $\epsilon$. Using then the compact embedding \(BV \subset L^1(\T)\), we obtain, up to taking a suitable subsequence, the following properties as \(\epsilon \to 0\):
\begin{align*}
    &\sup_{t \in [0, T]} J(\rho_\epsilon(t)) \leq F(\rho_0)\leq J(\rho_0)+C\epsilon; \\
    &\rho_\epsilon \leq \Vert \rho_0 \Vert_{L^\infty(\T)} \quad \text{a.e.\ in } (0, T) \times \T; \\
    &\rho_\epsilon \to \rho \quad \text{in } L^p ((0, T)\times\T) \quad \forall \;  p \in (1, \infty); \\
    &z_\epsilon \overset{\ast}{\rightharpoonup} z \quad \text{in } L^\infty ((0, T) \times \T); \\
    &(\Div z_\epsilon, \nabla \Div z_\epsilon) \rightharpoonup (\Div z, \nabla \Div z) \quad \text{in } L^2 ((0, T) \times \T) .
\end{align*}
The strong convergence in $L^p$ deserves a small clarification: first we extract, using Ascoli-Arzelà Theorem and the $W_2$ equicontinuity (which is standard for gradient flows) a subsequence which converges uniformly in time valued into $W_2$. This means in particular that for every $t\in[0,T]$ we have $\rho_\epsilon(t)\overset{\ast}{\rightharpoonup} \rho(t)$ as measures. Then, the uniform $BV$ bound and the compact injection into $L^1$ transforms 
this convergence into a strong $L^1$ convergence. Finally, the uniform $L^\infty$ bound makes this convergence strong in every $L^p$ space for $p<+\infty$. This is even stronger than what stated above, where only a convergence in space-time is underlined.

Moreover, for any \(u \in C^\infty_c([0, T) \times \T)\), 
\[
    \int_0^T \int_\T \left( \rho_\epsilon \partial_t u + \rho_\epsilon \nabla (\Div z_\epsilon - \epsilon f'(\rho_\epsilon)) \cdot \nabla u \right) \de x \de t= - \int_\T \rho_0(x) u(0, x) \de x.
\]

Thanks to the strong convergence of \(\rho_\epsilon\) and the weak convergence of \(\nabla \Div z_\epsilon\) in \(L^2((0, T) \times \T)\), we have \(\rho_\epsilon \nabla \Div z_\epsilon \rightharpoonup \rho \nabla \Div z\) in \(L^2((0, T) \times \T)\). Finally, thanks to the uniform upper bound and BV bound of \(\rho_\epsilon\) and \(\rho_\epsilon \geq \alpha\), we have \(\epsilon\rho_\epsilon \nabla f' (\rho_\epsilon) = \epsilon f''(\rho_\epsilon) \rho_\epsilon \nabla \rho_\epsilon \to 0\) in \(L^1 ((0, T) \times \T)\). Thus, we take \(\epsilon \to 0\), and get:
\[
    \int_0^T \int_\T \left( \rho \partial_t u + \rho \nabla  \Div z \cdot \nabla u \right) \de x \de t= - \int_\T \rho_0(x) u(0, x) \de x.
\]

It remains to prove \((\rho_t, z_t) \in \mathcal{A}\) a.e.\ \(t \in (0, T)\), which is a consequence of Lemma \ref{compazrho}.
\end{proof}

We now state a more general theorem where at the same time we both remove the assumption $\rho_0\in BV$ and give a more precise estimate on the decay of the BV norm of the solution.
\begin{theorem}
    Let \(\rho_0 \in \mathcal{P} (\T) \) with \( \alpha\leq\rho_0\leq\beta \) for two strictly positive constants $\alpha,\beta$. Then there exists a weak solution \(\rho\) of the TV-Wasserstein gradient flow satifying $\alpha\leq\rho\leq\beta$ and also satisfying the total variation decay:
    \begin{equation}
        \label{eq:TV_decay}
        J(\rho(t)) \leq \min\{A t^{-1} , B t^{-\frac{1}{3}} ,J(\rho_0)\}.
    \end{equation}
    where the constants $A,B$ only depend on $\kappa:=\frac\beta\alpha$.
\end{theorem}

\begin{proof}
    First, we assume that the initial data \(\rho_0\) is in \(BV(\T)\) and prove \eqref{eq:TV_decay}.
    
    Consider the solution of approximated TV-Wasserstein gradient flow \(\rho_\epsilon\) with the energy \(F_\epsilon\). We have the following regularity for \(\rho_\epsilon\) and \(z_\epsilon\):
    \begin{align*}
        &\rho_\epsilon \geq \alpha; \\
        &\rho_\epsilon \in L^2((0, T); H^1(\T)); \\
        &z_\epsilon \in L^2((0, T); H_{\Div}^2(\T)); \\
        & \Vert \nabla \Div z_\epsilon \Vert_{L^2} \leq \Vert \nabla \left( -\Div z_\epsilon + \epsilon f'(\rho_\epsilon) \right)\Vert_{L^2}.
    \end{align*}

   We now consider the dissipation, along the approximated TV-Wasserstein gradient flow, of the TV energy itself, and we obtain for any \(t \geq s \geq 0 \),
    \begin{align*}
     J(\rho_\epsilon(s)) - J(\rho_\epsilon(t))  &= 
     \int_s^t \int_\T \langle \partial_t \rho_\epsilon, -\Div z_\epsilon \rangle \\
     &=\int_s^t \int_\T \rho_\epsilon \nabla \left( -\Div z_\epsilon + \epsilon f'(\rho_\epsilon) \right) \cdot \nabla ( -\Div z_\epsilon)\\
     &=\int_s^t\int_\T\rho_\varepsilon   \vert \nabla \left( -\Div z_\epsilon  \right) \vert^2-\int_s^t\int_\T\varepsilon\rho_\varepsilon\nabla (  f'(\rho_\epsilon))\cdot  \nabla (\Div z_\epsilon )\\ 
        &\geq \int_s^t \int_\T  \rho_\varepsilon\vert \nabla \Div z_\epsilon \vert^2\\
         &\geq \alpha\int_s^t \int_\T  \vert \nabla \Div z_\epsilon \vert^2.
    \end{align*}
    The first equality above is justified by the fact that $\partial_t \rho_\varepsilon$ belongs to $L^2([0,T];H^{-1}(\T))$ and $-\Div z_\epsilon(t)$ is, for a.e.\ $t$, an element of $\partial_{L^2}J(\rho_\varepsilon(t))$ and as a function of time, it belongs to $L^2([0,T];H^{1}(\T))$, which makes the computation of the derivative of $J$ along the curve rigorous; the next inequality is justified by Lemma \ref{lem:inequality_rho_div_z} and the last one by $\rho_\varepsilon\geq\alpha$.

    We now use  the Poincar\'e inequality on $\Div z_\epsilon$, which has zero-mean. The characterization of the optimal constant on the torus provides, for every zero-mean function, the bound $4\pi^2\int_\T v^2\leq\int_\T |\nabla v|^2$, so that we have 
    \begin{align*}
      J(\rho_\epsilon(s))  - J(\rho_\epsilon(t))  &\geq 4\pi^2\alpha \int_s^t \int_\T  \vert \Div z_\epsilon \vert^2 \\
        &\geq 4\pi^2\frac{\alpha}{\beta}\int_s^t \int_\T\rho_\epsilon\vert \Div z_\epsilon \vert^2.
    \end{align*}
    We now use the following computation based on the fact that $\rho_\epsilon$ has unit mass:
    \begin{equation}\label{qor2}
        \int_\T\rho_\epsilon\vert \Div z_\epsilon \vert^q=\int_\T\rho_\epsilon\vert \Div z_\epsilon \vert^q \left(\int_\T\rho_\epsilon\right)^{q-1}\geq \left(\int_\T-\rho_\epsilon\Div z_\epsilon\right)^q=J(\rho_\epsilon)^q.
        \end{equation}
    Using the above inequality for $q=2$ we obtain
    \[
   J(\rho_\epsilon(s)) - J(\rho_\epsilon(t)) \geq \frac{4\pi^2}{\kappa}\int_s^tJ(\rho_\epsilon)^2.
    \]
    If, instead of using the Poincar\'e inequality, we use a Gagliardo-Nirenberg interpolation type inequality 
    \[
        9\Vert z_\epsilon \Vert^2_{L^\infty} \int \vert \nabla \Div z_\epsilon \vert^2 \geq \int \vert \Div z_\epsilon \vert^4,
    \] 
    which is proven in Appendix \ref{sec:appendix_inequality}, combined with \eqref{qor2} for $q=4$, we similarly obtain
$$
       J(\rho_\epsilon(s)) -J(\rho_\epsilon(t))  \geq \frac{\alpha}{9\beta}\int_s^t \int_\T \rho_\epsilon \vert \Div z_\epsilon \vert^4 \geq \frac{1}{9\kappa}\int_s^t J^4(\rho_\epsilon).
$$

    We now want to pass to the limit in these inequalities when $\epsilon\to 0$. The lower semicontinuity is enough to treat the right-hand side and as well as $J(\rho_\epsilon(t))$, but we need more in order to deal with $J(\rho_\epsilon(s))$.

    We therefore fix $s_0<s_1<t$ and we average the above inequalities for $s\in (s_0,s_1)$. If we denote by $\eta$ the function which takes value $1$ on $(s_1,t)$ and $\eta(r)=\frac{(r-s_0)}{s_1-s_0}$ for $r\in (s_0,s_1)$, we obtain
    \[
        \frac{1}{s_1-s_0}\int_{s_0}^{s_1}J(\rho_\epsilon(s))ds +C\epsilon\geq J(\rho_\epsilon(t))+c\int_{s_0}^t \eta(r)J^q(\rho_\epsilon(r))dr,
    \]
    where $c$ can take the value $4\pi^2\frac\alpha\beta$ or $\frac{\alpha}{9\beta}$ and $q$ the value $2 $ or $4$ in the two cases of interest.

    We then pass to the limit $\epsilon\to 0$ and use
    \[
        \int_{s_0}^{s_1}\!\!J(\rho_\epsilon(s))ds=\!\int_{s_0}^{s_1}\!\!\int_\T -\rho_\epsilon(s)\Div(z_\epsilon(s))\to \int_{s_0}^{s_1}\!\!\int_\T -\rho(s)\Div(z(s))=\!\int_{s_0}^{s_1}\!\!J(\rho(s))ds,
    \]
    where the convergence is justified by the strong $L^2$ convergence for $\rho_\epsilon$ and the weak $L^2$ convergence for $\Div z_\epsilon$ (these convergences only hold in space-time, which explains why we need to take the average in time). We then obtain, also using lower-semicontinuity,
    \[
        \frac{1}{s_1-s_0}\int_{s_0}^{s_1}J(\rho(s))ds \geq J(\rho(t))+c\int_{s_0}^t \eta(r)J^q(\rho(r))dr.
    \]
    Then, considering a point $s_0$ which is a Lebesgue point for $s\mapsto J(\rho(s))$ and taking the limit $s_1\to s_0$, we finally obtain
    \[
        J(\rho(s_0)) \geq J(\rho(t))+c\int_{s_0}^t J^q(\rho(r))dr.
    \]
    This proves that the function $s\mapsto J(\rho(s))$ is Lebesgue-equivalent to a nonincreasing function $\phi$ satisfying the inequality $\phi'\leq -c\phi^q$ a.e.\
    If we consider the function $\psi:=\phi^{1-q}$, such a function is nondecreasing (since in both cases $q>1$) and it satisfies $\psi'=(1-q)\phi^{-q}\phi'\geq c(q-1)$ a.e.\ Since monotone increasing functions grow more than the primitive of their a.e.\ derivative, and using the non-negativity of $\psi$, we obtain $\psi(t)\geq c(q-1)t$, i.e.\ $\phi(t)\leq [c(q-1)t]^{-1/(q-1)}$. By lower-semicontinuity, the same inequality is true for $J(\rho(t))$ for every $t$, and not only for a.e.\ $t$.

    This proves
    \[
        J(\rho(t))\leq \min\left\{At^{-1},Bt^{-1/3}\right\},
    \]
    with $A=\frac{\kappa}{4\pi^2}$ and $B=3\kappa^{1/3}$. The bound $J(\rho(t))\leq J(\rho_0)$ is also true by construction, by taking the limit $\epsilon\to 0$ of the inequality
    \[
        J(\rho_\epsilon(t))+\epsilon\int_\T f(\rho_\epsilon(t))\leq J(\rho_0)+\epsilon\int_\T f(\rho_0).
    \]

    Next we consider the case when the initial datum \(\rho_0\) is not BV. To do this, we can approximate \(\rho_0\) by a sequence of \(BV\) functions \(\rho_0^n\) satisfying the same bounds $\alpha\leq\rho_0^n\leq\beta$.

    Let \(( \rho^n, \Div z^n )\) be the weak solution of the TV-Wasserstein gradient flow with the initial data \(\rho_0^n\) obtained above, for the case where the initial datum is BV. They were obtained as a limit of the corresponding approximated TV gradient flow. We note that the estimates that we provided on such solutions of the approximated TV gradient flow also imply a bound on the quantities
    $$\int_s^T\int_\T|\nabla \Div z^n|^2\de x \de t,\quad\int_s^T|(\rho^n)'|_{W_2}^2(t)\de t$$
    in terms of $J(\rho^n(s))$, which, itself, is bounded by $s^{-1/3}$. 

    This allows to pass to the limit the weak formulation of the equation, at least on the open set $(0,\infty)$: for any \(\varphi \in C^\infty_c((0, \infty) \times \T)\), we have
    \begin{equation}
        \label{eq:weak_form_approximated_initial_data}
        \int_0^\infty \int_\T \left( \rho^n \partial_t u + \rho^n \nabla  \Div z^n \cdot \nabla u \right) \de x \de t= 0.
    \end{equation}
    Up to a subsequence, we may assume that there exists \((\rho, \Div z)\) which satisfies the following properties:
    \begin{align*}
        &(\rho^n, \Div z^n) \to (\rho, \Div z) \quad \text{a.e.\ in } (0, \infty) \times \T; \\
        \forall \; [T_1, T_2]\times \T \; &\text{, with} \; 0 < T_1 < T_2 < \infty: \\ 
        &\rho^n \to \rho \quad \text{in } L^p \quad \forall \;  p \in (1, \infty); \\
        &z_n \overset{\ast}{\rightharpoonup} z \quad \text{in } L^\infty ; \\
        &(\Div z^n, \nabla \Div z^n) \rightharpoonup (\Div z, \nabla \Div z) \quad \text{in } L^2 .
    \end{align*}
    Using the same argument as in \eqref{eq:J_compatibility}, we have \((\rho_t, z_t) \in \mathcal{A}\) a.e. \(t \in (0, T)\), and this proves that we found a solution to the TV gradient flow equation for $t>0$, and of course this solution satisfies all the bounds in the claim (both upper and lower bounds on the density, and $J(\rho_t)\leq\min\{ At^{-1},Bt^{-1/3}\})$.

    We are now left with proving that the initial datum is indeed preserved for \(\rho\). If we can prove equicontinuity of the curves $\rho^n$ than this follows from uniform convergence (the curves $\rho^n$ are continuous on $[0,T]$ and uniformly converge to $\rho$, which will be continuous : moreover the values $\rho^n(0)$ converge to $\rho(0)$ but, by construction, also to $\rho_0$). We use
    $$\int_{s_0}^{T}|(\rho^n)'|^2(t)dt\leq Cs_0^{-1/3}, $$
    and apply Lemma \ref{equicgamma} with \(q = 2\) and $\gamma=1/3$, which proves that $\rho^n$ is uniformly $C^{0, 1/3}$ and hence the sequence is equi-continuous. It is crucial here that we use $\gamma=1/3$ and not $\gamma=1$.
    \end{proof}

\appendix

\section{Useful inequalities}
\label{sec:appendix_inequality}

We gather here the proofs of some inequalities and estmates used in the last section.
\begin{lemma} 
   The following Gagliardo-Nirenberg type inequality holds:
    \[
    9 \Vert z \Vert^2_{L^\infty} \int_\T \vert \nabla \Div z \vert^2 \geq \int_\T \vert \Div z \vert^4 \quad \forall z \in H^2_{\Div}(\T).
    \] 
\end{lemma}

\begin{proof}
For simplicity we set $K=\Div z$. We consider
$$\int_\T K^4=\int_\T K^3\Div z=-\int_\T z\cdot 3K^2\nabla K\leq 3||z||_{L^\infty}\left(\int_\T|\nabla K|^2\right)^{1/2}\left(\int_\T K^4\right)^{1/2}.$$
We then divide by $\left(\int_\T K^4\right)^{1/2}$ and square both sides.
\end{proof}

\begin{lemma}\label{equicgamma}
Assume that a continuous curve $\omega:[0,T]\to X$ valued in a metric space $(X,d)$ satisfies
\[
    \int_{s}^{T}|\omega'|^q(t)\de t\leq Cs^{-\gamma}
\]
for an exponent $0<\gamma<q-1$ and for every $s>0$. Then $\omega$ has the following H\"older modulus of continuity
\[
    d(\omega(s),\omega(t))\leq C'|t-s|^{1-\frac{\gamma+1}{q}}
\]
for a constant $C'$ only depending on $C,\gamma$ and $q$.
\end{lemma}
\begin{proof}
First, we observe that for every \(s<s'\) we have 
\[
    d(\omega(s),\omega(s'))\leq\int_s^{s'}|\omega'(t)|\de t\leq \left(\int_s^{s'}|\omega'(t)|^q \de t\right)^{1/q}|s'-s|^{1/q'}\leq C\frac{|s-s'|^{1/q'}}{s^{\gamma/q}},
\]
where \(q'=\frac{q}{q-1}\) is the conjugate exponent of \(q\).

Let us now take $0<s<s_0$ and define $s_k:=2^{-k}s_0$. There exists $N\geq 0$ such that $s_{N+1}=s_N/2\leq s<s_N$.  

We then use
\begin{eqnarray*}
    d(\omega(s),\omega(s_0))&\leq& d(\omega(s),\omega(s_N))+\sum_{k=0}^{N-1}d(\omega(s_k),\omega(s_{k+1}))\\
    &\leq& C\frac{|s-s_N|^{1/q'}}{s^{\gamma/q}}+C\frac{|s_0-s_1|^{1/q'}}{s_1^{\gamma/q}}\chi_{N>1}\sum_{k=0}^{N-1}2^{-k(1-\frac{\gamma+1}{q})}.
\end{eqnarray*}
Note that the last sum may be empty in case $N=0$, which explains the factor $\chi_{N>1}$ which is either one if $N>1$ or zero if $N=0$. Anyway, the sum converges because we have $1-\frac{\gamma+1}{q}>0$, so that we obtain
\[
    d(\omega(s),\omega(s_0))\leq C\frac{|s-s_N|^{1/q'}}{s^{\gamma/q}}+\chi_{N>1}C'\frac{|s_0-s_1|^{1/q'}}{s_1^{\gamma/q}}, 
\]
where the second term only appears if $N>0$.
We then use $|s-s_N|\leq s$ and $|s_0-s_1|\leq s_1$ to obtain
$$d(\omega(s),\omega(s_0))\leq C|s-s_N|^{1-\frac{\gamma+1}{q}}+C'|s_0-s_1|^{1-\frac{\gamma+1}{q}}\chi_{N>1},$$
and we conclude by observing that we have both $|s-s_N|\leq |s-s_0|$ and (when $N>1$) $|s_1-s_0|\leq |s-s_0|$. This provides the desired estimate for $s>0$. The case $s=0$ can be simply obtained from the case $s>0$ by continuity.
\end{proof}

\begin{remark}
    We note that the result of the previous lemma is sharp, in the sense that the modulus of continuity fo $\omega$ cannot in general be improved, as one can see from the example of the curve $\omega: [0,1]\to [0,1]$ given by $\omega(t)=t^{1-\frac{\gamma+1}{q}}$.
\end{remark}

\section*{Acknowledgements}

This work was supported by the European Union via the ERC AdG 101054420 EYAWKAJKOS project.

\bibliographystyle{abbrv}
\bibliography{./biblio}

\end{document}